\documentclass[a4paper,9pt]{article}
\usepackage{amsfonts,amsmath,amssymb}
\usepackage{amsthm}
\usepackage{geometry}
\usepackage{graphicx}
\usepackage{hyperref} 
\usepackage{mathrsfs}  
\usepackage{tikz}
\usetikzlibrary{math}
\usepackage{multicol}
\usepackage{enumitem}
\usepackage{multirow}

\newtheorem{theorem}{Theorem}[section]
\newtheorem{lemma}{Lemma}[section]
\newtheorem{corollary}{Corollary}[section]

\title{The Locating-Chromatic number of an $n$-ary Trees} 

\author
{Yusuf Hafidh $^{1}$, Edy Tri Baskoro $^{1,2,*}$, and Devi Imulia Dian Primaskun $^{1}$\\
	\\
	\normalsize{$^{1}$ \ Combinatorial Mathematics Research Group, Faculty of Mathematics }\\
        \normalsize and Natural Sciences, Institut Teknologi Bandung, Indonesia
\\
        \normalsize{$^{2}$ \ Center for Research Collaboration on Graph Theory and Combinatorics, Indonesia}\\
        \\
	\normalsize{Emails: yusuf.hafidh@itb.ac.id, ebaskoro@itb.ac.id, and imuliadevi@gmail.com}\\
}

\date{}
\begin{document}
\baselineskip20pt
\maketitle

\begin{abstract}
	The locating-chromatic number of a graph $G$ is the smallest integer $n$, such that $G$ has a proper $n$-coloring $c$ and all vertices have different vectors of distances to the colors generated by $c$.
    We study the asymptotic value of the locating-chromatic number of a $k$-level $n$-ary tree. The locating-chromatic number of this tree acts very differently when $k$ goes to infinity and when $n$ goes to infinity.
    If we fix $k\geq2$, almost all $n$-ary Tree $T(n,k)$ satisfy $\chi_L(T(n,k))=n+k-1$; so $\lim\limits_{n\to \infty} \chi_L(T(n,k))-n=k-1$. But if we fix $n\geq 2$, then $\chi_L(T(n,k))=o(k)$.
\end{abstract}

Keywords: $n$-ary Tree, locating-chromatic number, graph.

Mathematics Subject Classification: 05C12, 05C63, 05C15

\section{Introduction}

In general, the locating-chromatic number is not an increasing function, that is a subgraph does not always have a smaller locating-chromatic number than its supergraph. However, subgraphs of binary trees $T(2,k)$ have smaller locating-chromatic numbers than its binary tree supergraph \cite{sofyan}. Although the increasing property has not been generalized for any $n$-ary tree $T(n,k)$ for $n\geq3$, the locating-chromatic number of complete $n$-ary trees will contribute to the locating chromatic number of trees in general, because any tree $G$ is a subgraph of a complete $n$-ary tree $T(n,k)$ where $n=\Delta(G)$ and $k=\lceil diam(G)/2) \rceil$. Algorithms to optimize the locating-coloring of trees can be found in \cite{hilda, devi, Beh15, chr, DAM19, Hafidh2022, SLT75}. Locating chromatic number of specific graphs can be found in \cite{Beh11, AJC09}.

In this paper, we consider simple, undirected, and connected graphs. For general graph definitions, see \cite{Book17}. For $k\geq 1$, a \textit{proper k-coloring} of $G$ is a vertex $k$-coloring of $G$ such that every edge in $G$ joins two different colors. A color code of vertex $v$ in $G$ with respect to $c$, denoted by $a_c(v)$, is a vector of distances from $v$ to the nearest vertex of each color. If all vertices of $G$ have distinct color codes, then $c$ is called a locating $k$-coloring of $G$. The locating-chromatic number of $G$, denoted by $\chi_L(G)$, is the smallest positive integer $k$ such that $G$ has a locating $k-$coloring.

For integers $n\geq2$ and $k\geq 1$, we define an $n$-ary Tree $T(n,k)$ recursively with $T(n,1)=S_n$, and for $k\geq2$, $T(n,k)$ is obtained by connecting a new vertex to the center of $n$ disjoint $T(n,k-1)$, see Figure \ref{T43}.
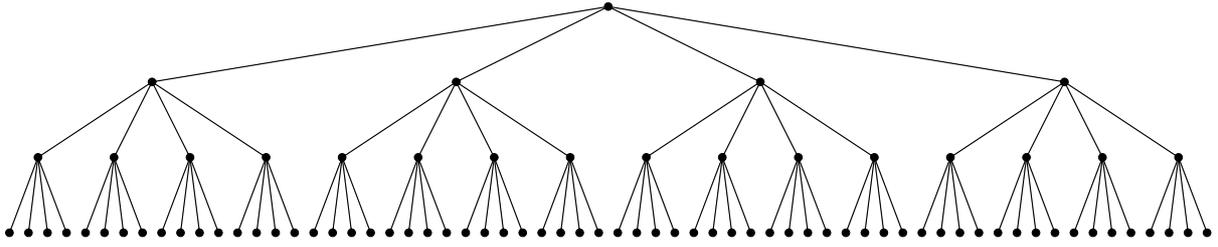
\begin{figure}[!h]
	\centering
		\begin{tikzpicture}[xscale=0.2]
		\def\n {4}
		\def\k {3}
		\tikzmath{\b=\n^\k/2;}
		\draw[fill] (\b,\k+1)circle(0.2 and 0.05);
		\foreach \y in {1,...,\k}
		{
			\tikzmath{\a=\n^(\k-\y+1);\b=\n^(\y-1);}
			\foreach \x in {1,...,\a}
			\tikzmath{\m=mod(\x-1,\n)+1;}
			\draw[fill] (\b*\x+\n*\b*0.5-\m*\b,\y+1)--(\b*\x-\b*0.5,\y)circle(0.2 and 0.05);
		}
		\end{tikzpicture}
		\caption{Graph $T(4,3)$}
		\label{T43}
\end{figure}


Welyyanti \textit{et al.} have proposed locating-chromatic number of complete $n-$ary trees $T(n,k)$ for $k=1,2,3$ and the upper bound for $n,k \geq 3$ in the following theorem.

\begin{theorem}
	\label{tnk13}
	\cite{welly}
	For integers $n\geq 2$ and $k\geq1$,
	\begin{align*}
		\chi_L(T(n,k))&=
		\begin{cases}
		n+1,\ \phantom{-1}\quad \text{for $k=1$}\\
		n+k-1,\quad \text{for $k=2,3$}
		\end{cases}\qquad \text{and}\\
		\chi_L(T(n,k))&\leq n+k-1,\qquad \text{for $k\geq 4$.}
	\end{align*}
\end{theorem}

\section{Main Result}
	In this paper, we study the asymptotic properties of the locating-chromatic number of a $k$-level $n$-ary tree.
	First, we prove that for every integer $k\geq 4$, the upper bound given in Theorem \ref{tnk13} is achieved by almost all graphs $T(n,k)$.
	
	\begin{theorem}\label{Mk}
		For any fixed integer $k\geq 4$, almost all $n$-ary Tree $T(n,k)$ with $n\geq 2$ statisfy $\chi_L(T(n,k))=n+k-1$; in other words, there are only finitely many integers $n$ such that $\chi_L(T(n,k))<n+k-1$.
	\end{theorem}
	\begin{proof}
		To prove this theorem, we will prove that for every integer $k\geq 4$, there exists an integer $M_k$ such that if $n$ is an integer with $n\geq M_k$, then $\chi_L(T(n,k))=n+k-1$.
		Since $(n+k-2)\binom{n+k-3}{n}$ is a polynomial of order $k-2$ in $n$, then 
		\begin{align}\label{lim}
		\lim\limits_{n\to\infty}\frac{n^{k-1}}{(n+k-2)\binom{n+k-3}{n}}=\infty.
		\end{align}
		Since the limit of (\ref{lim}) is infinite then there exists an integer $M_k$ such that for $n\geq M_k$, we have $$\displaystyle \frac{n^{k-1}}{(n+k-2)\binom{n+k-3}{n}}>(2k)^{k-3}.$$ By Theorem \ref{tnk13}, $\chi_L(T(n,k)) \leq n+k-1,$ for $n \geq 4$. To prove that $\chi_L(T(n,k))=n+k-1$ for $n\geq M_k$, we only need to show that there is no locating $(n+k-2)$-coloring of $T(n,k)$. Suppose otherwise, let $c$ be a locating $(n+k-2)$-coloring of $T(n,k)$.
		
		Since $c$ is a locating coloring, all vertices in any palm of $T(n,k)$ will receive distinct colors. We say that two palms have the same coloring type if their local end-branches have the same color and their leaves have the same colors. The number of coloring types for all palms is $(n+k-2)\binom{n+k-3}{n}$.
		
		Since there are $n^{k-1}$ palms and $(n+k-2)\binom{n+k-3}{n}$ coloring types in $T(n,k)$, there is at least $$\displaystyle \left\lceil\frac{n^{k-1}}{(n+k-2)\binom{n+k-3}{n}}\right\rceil\geq(2k)^{k-3}+1$$ palms with the same coloring type. Without loss of generality, let these palms with the same coloring type have color one in the local end-branch and colors $2,3, \cdots, n+1$ in its leaves.
		For $i=1,2,\cdots, n+k-2$, let $a_i=d(v,C_i)$. Since the diameter of $T(n,k)$ is $2k$, then the color code of $v$ is $(0,1,\cdots,1,a_{n+2},\cdots,a_{n+k-2})$ with $a_i \leq 2k$ for $i=n+2,\cdots,n+k-2$. This means that there are at most $(2k)^{k-3}$ possible color codes for $v$. Since there are at least $(2k)^{k-3}+1$ palms in this coloring type and the number of possible color codes for its local-end branch is at most $(2k)^{k-3}$, then there are at least two local-end branches with the same color code, a contradiction. Therefore, for every $n\geq M_k$, we have $\chi_L(T(n,k))=n+k-1$.
	\end{proof}
	
	Another way of seeing Theorem \ref{Mk} is if we set $k\geq 4$ to be a fixed integer then when $n$ goes to infinity, the value of $\chi_L(T(n,k))$ goes to $n+k-1$.
	\begin{corollary}
		For any integer $k\geq4$,
		\[\lim\limits_{n\to\infty}\chi_L(T(n,k))-n=k-1.\]
	\end{corollary} 

	Now, we study the value of $\chi_L{T(n,k)}$ when $k\to \infty$. For two functions $f,g:\mathbb{N}\to \mathbb{N}$, we say $f(k)=o(g(k))$ if and only if $\displaystyle \lim\limits_{k\to \infty} \frac{f(k)}{g(k)}=0$.
	
	\begin{theorem}\label{ok}
		For any fixed integer $n\geq 3$, $\chi_L(T(n,k))=o(k)$.
	\end{theorem}
	
	To prove Theorem \ref{ok}, we need lemma \ref{k+t} below.
	Let $D_t(v)$ be the set of vertices with distance $t$ to $v$. Suppose that $v$ is the center of $T(n,k)$, a coloring $c$ of $T(n,k)$ is said to have a $c^t_{a,A}$ property if the following holds;
	\begin{enumerate}
		\item $c(v)=a\notin A$,
		\item $c(D_t(v))=A$ and $|A|=n^t$.
		\item For $1\leq i\leq t-1$, $c(D_i(v))=\{1,2\}$.
	\end{enumerate}

	\begin{lemma}\label{k+t}
		Suppose that $T(n,k)$ has a locating $m$-coloring with $c^t_{a,A}$ property, then $T(n,k+t)$ has a locating $(m+2)$-coloring with $c^t_{a',A}$ property.
	\end{lemma}
	\begin{proof}
		Let $c$ be a locating $m$-coloring of $T(n,k)$ having $c^t_{a,A}$ property. Without loss of generality, let $a=m$ and $A=\{3,4,\cdots,n^t+2\}$.
		For $i=1,2,\cdots,n^t$, let $T_i$ be the $i^{th}$ copy of $T(n,k)$ in $T(n,k+t)$ as illustrated in figure \ref{Tnk+t}
		
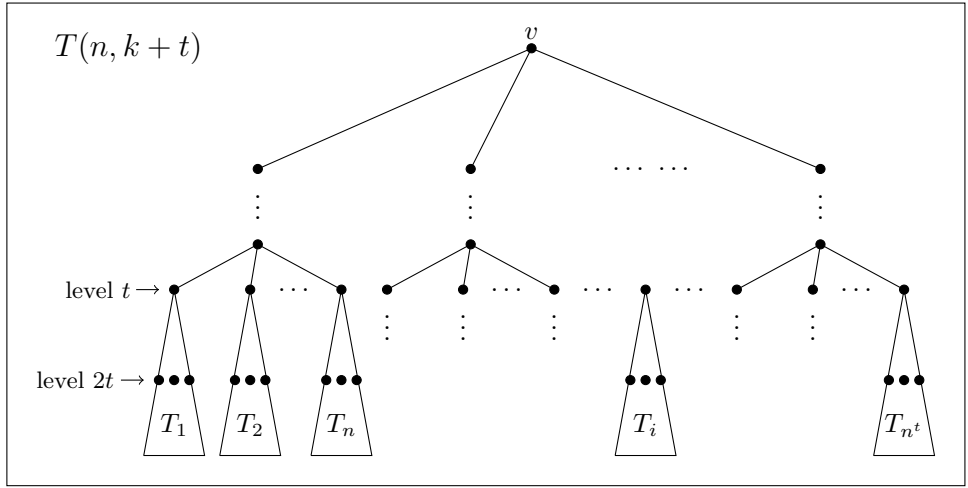
\begin{figure}[!h]
	\begin{center}
			\begin{tikzpicture}[scale=0.2]
    	\draw (-8,-7)rectangle(55,25);
    	\draw (0,22)node{\large ${T(n,k+t)}$};
    	
		\draw[fill=black] (3,6)circle(0.3);
		\draw[fill=black] (8,6)circle(0.3);
		\draw[fill=black] (14,6)circle(0.3);
		\draw[fill=black] (17,6)circle(0.3);
		\draw[fill=black] (22,6)circle(0.3);
		\draw[fill=black] (28,6)circle(0.3);
		\draw[fill=black] (34,6)circle(0.3);
		\draw[fill=black] (40,6)circle(0.3);
		\draw[fill=black] (45,6)circle(0.3);
		\draw[fill=black] (51,6)circle(0.3);
		
		\node at (11,6){\ldots};
		\node at (25,6){\ldots};
		\node at (37,6){\ldots};
		\node at (48,6){\ldots};
		\node at (31,6){\ldots};
		\node at (33,14){\ldots};
		\node at (36,14){\ldots};

		\draw (1,-5)--(5,-5)--(3,6)--(1,-5);
		\draw (6,-5)--(10,-5)--(8,6)--(6,-5);
		\draw (12,-5)--(16,-5)--(14,6)--(12,-5);
		
		\draw (32,-5)--(36,-5)--(34,6)--(32,-5);
		\draw (49,-5)--(53,-5)--(51,6)--(49,-5);
		
		\draw (3,6)--(8.5,9);
		\draw (8,6)--(8.5,9);
		\draw (14,6)--(8.5,9);
		
		\draw (17,6)--(22.5,9); 
		\draw (22,6)--(22.5,9);
		\draw (28,6)--(22.5,9);
		
		\draw (40,6)--(45.5,9);
		\draw (45,6)--(45.5,9);
		\draw (51,6)--(45.5,9);
		
		\draw[fill=black] (8.5,9)circle(0.3);
		\draw[fill=black] (22.5,9)circle(0.3);
		\draw[fill=black] (45.5,9)circle(0.3);
		\draw[fill=black] (8.5,14)circle(0.3);
		\draw[fill=black] (22.5,14)circle(0.3);
		\draw[fill=black] (45.5,14)circle(0.3);
		
		\draw[fill=black] (26.5,22)circle(0.3)node[above]{$v$};
		\draw (22.5,14)--(26.5,22);
		\draw (8.5,14)--(26.5,22);
		\draw (45.5,14)--(26.5,22);
		
		\node at (8.5,12){\vdots};
		\node at (22.5,12){\vdots};
		\node at (45.5,12){\vdots};
		
		\node at (17,4){\vdots};
		\node at (22,4){\vdots};
		\node at (28,4){\vdots};
		\node at (40,4){\vdots};
		\node at (45,4){\vdots};

		\draw[fill=black] (2,0)circle(0.3);
		\draw[fill=black] (3,0)circle(0.3);
		\draw[fill=black] (4,0)circle(0.3);
		\draw[fill=black] (7,0)circle(0.3);
		\draw[fill=black] (8,0)circle(0.3);
		\draw[fill=black] (9,0)circle(0.3);
		\draw[fill=black] (13,0)circle(0.3);
		\draw[fill=black] (14,0)circle(0.3);
		\draw[fill=black] (15,0)circle(0.3);
		\draw[fill=black] (33,0)circle(0.3);
		\draw[fill=black] (34,0)circle(0.3);
		\draw[fill=black] (35,0)circle(0.3);
		\draw[fill=black] (50,0)circle(0.3);
		\draw[fill=black] (51,0)circle(0.3);
		\draw[fill=black] (52,0)circle(0.3);
		
		\node at (-3.5,0){\small level $2t$};
		\draw[->] (-0.5,0)--(1,0);
		\node at (-2,6){\small level $t$};
		\draw[->] (0.5,6)--(2,6);
		
		\node at (3,-3){$T_1$};
		\node at (8,-3){$T_2$};
		\node at (14,-3){$T_n$};
		\node at (34,-3){$T_i$};
		\node at (51,-3){$T_{n^t}$};
	\end{tikzpicture}
		\caption{Graph $T(n,k+t)$\label{Tnk+t}}
	\end{center}
\end{figure}
	
	\noindent Now we construct a locating $(m+2)$-coloring for $T(n,k+t)$.
	\begin{enumerate}
	    \item Let $v$ be the center of $T(n,k+t)$ and color $v$ with $m+2$.
	    \item For $i=1,2,\cdots,t$; color all vertices in $D_i(v)$ with $i \pmod 2 +1$.
	    \item For $i=1,2,\cdots,n^t$; color $T_i$ with colors in  $\left\{1,2,\cdots,m+1\right\}-$ \linebreak $\left\{i+n^{t-1}-1\pmod{n^t} +3\right\}$ by permuting the coloring $c$ of $T(n,k)$ such that the coloring have $c^t_{i+2,B}$ property, where\\ $B=\left\{3,4,\cdots,n^t+4\right\}-\left\{i+2,i+n^{t-1}-1\pmod{n^t} +3\right\}$.
	\end{enumerate}
	
    \noindent Here are some important properties of the previous coloring.
    \begin{enumerate}[label=(\arabic*)]
        \item The vertex $v$ is the only vertex in $T(n,k+t)$ to be colored $m+2$.
        \item All vertices in $D_i(v)$ for $i=1,2,\cdots,t-1,t+1,\cdots 2t-1$; is colored by $1$ and $2$.
        \item All vertices in $D_t(v)$ has different colors from $3,4,\cdots,n^t+2$.
        \item $T_i$ is missing color $i+n^{t-1}-1\pmod{n^t}+3$ which is used in every other $T_j$, either in the center or in the level $2t$.
    \end{enumerate}
    
    To prove that the coloring is a locating coloring, suppose there are vertices $u$ and $w$ with the same color code. 
    Because of property (1), $u$ and $w$ are at the same level. Since two vertices in level $1,2,\cdots,t$ are distinguished by colors in level $t$, and two vertices in the same $T_i$ are distinguished because $c$ is a locating coloring, then $u$ and $w$ must be in the same level at different $T_i$'s.
    
    Let $u$ and $w$ be in level $r$, $u$ in $T_i$ and $w$ in $T_j$. Since $T_i$ does not contain color $a=i+n^{t-1}-1\pmod{n^t}+3$, and the nearest vertex with color $a$ is in level $2t$ in $T_{i-1}$ or $T_{i+1}$, then $d(u,C_a)=(r-t)+2+t$. On the other hand, $T_j$ contains color $a$ either in level $t$ or $2t$, then $d(w,C_a)\leq(r-t)+t$ and hence $d(w,C_a)< d(u,C_a)$.
	\end{proof}
    
    \begin{proof}[Proof of Theorem \ref{ok}]
        We will prove that for every $\varepsilon$, there is an integer $K$ such that $k>K$ implies $\displaystyle\frac{\chi_L(T(n,k))}{k}<\varepsilon$.
        
        Let $t$ be an integer such that $\displaystyle\frac{2}{t}<\frac{\varepsilon}{2}$. For every $i\in\{t,t+1,\cdots,2t-1\}$, $T(n,i)$ has a locating coloring with $c^t_{a,A}$ property. This could easily be achieved by coloring all vertices in level $0$, $t$, and $i$ with different colors. Therefore $\chi_L{T(n,i)}\leq f(n,t):=1+n^t+n^{2t-1}$.
        Since $f(n,t)$ is not a function of $k$, there is an integer $K$ such that $k>K$ implies \begin{align}
            \frac{f(n,t)}{k}<\frac{\varepsilon}{2}.\label{1}
        \end{align}
        
        Consider $\chi_L(T(n,k))$ for $k>K$. Write $k=at+i$ where $i\in\{t,t+1,\cdots,2t-1\}$. From Lemma \ref{k+t} and (\ref{1}), $$\frac{\chi_L(T(n,k))}{k}\leq \frac{2a+f(n,t)}{k}=\frac{\frac{2(k-i)}{t}+f(n,t)}{k}\leq \frac{2}{t}+\frac{f(n,t)}{k}<\varepsilon.$$
    \end{proof}

	\section*{Acknowledgment}
	This research has been supported by "\textit{PPMI FMIPA ITB}" managed by Faculty of Mathematics and Natural Sciences, Institut Teknologi Bandung.

\bibliographystyle{plain}
\bibliography{ref}

\end{document}